\documentclass[10pt]{amsart}
\usepackage{amssymb}
\usepackage{bm}
\usepackage{graphicx}
\usepackage[centertags]{amsmath}
\usepackage{amsfonts}
\usepackage{amsthm}
\usepackage{graphicx}
\newtheorem{thm}{Theorem}

\newtheorem{lem}[thm]{Lemma}

\newtheorem{defn}[thm]{Definition}
\theoremstyle{definition}

\newcommand{\nn}{\mathbb{N}}

\newcommand{\con}{\smallfrown}
\newcommand{\meg}{\geqslant}
\newcommand{\mik}{\leqslant}

\newcommand{\supp}{\mathrm{supp}}

\begin{document}
\title[On colorings of variable words]{On colorings of variable words}

\author{Konstantinos Tyros}

\address{Mathematics Institute, University of Warwick, Coventry, CV4 7AL, UK}
\email{k.tyros@warwick.ac.uk}

\thanks{Supported by ERC grant 306493}
\maketitle

\begin{abstract}
  In this note, we prove that the base case of the Graham--Rothschild Theorem, i.e., the one that considers colorings of the ($1$-dimensional) variable words, admits bounds in the class $\mathcal{E}^5$ of Grzegorczyk's hierarchy.
\end{abstract}

\section{Introduction}
The Graham--Rothschild Theorem \cite{GR} is a generalization of the well know Hales--Jewett Theorem that considers colorings of $m$-parameter sets instead of constant words. The best known bounds
for the Graham--Rothschild Theorem are due to S. Shelah \cite{Sh} and belong to the class $\mathcal{E}^6$ of Grzegorczyk's hierarchy. In this note we consider the ``base'' case of the Graham--Rothschild Theorem, that concerns colorings of ($1$-dimensional) variable words.
We obtain bounds for this base case in $\mathcal{E}^5$ of Grzegorczyk's hierarchy.
Although the proof is an appropriate modification of S. Shelah's prood for the Hales--Jewett Theorem,
it is streamline and independent.

The base case of the Graham--Rothschild Theorem is of particular interest, since it 
is the one needed for the proof of the density Hales--Jewett Theorem in \cite{DKT}.
Moreover, it
has as an immediate consequence the finite version of the Carlson--Simpson Theorem on the left variable words and therefore the
finite version of the Halpern--L\"{a}uchli theorem for level products of homogeneous trees (see also \cite{Soc}).

To state the result of this note, we need some pieces of notation.
Let $k$ and $n$ be positive integers. By $[k]$ we denote the set $\{1,...,k\}$ and $[k]^n$ the set of all sequences $(a_0,...,a_{n-1})$ of length $n$ taking values in $[k]$. We view $[k]$ as a finite alphabet and the elements of $[k]^n$ as words. Thus, by the term \emph{word over $k$ of length $n$} we mean an element of $[k]^n$. Also let $m$ be a positive integer and  $v,v_0,...,v_{m-1}$ distinct symbols not belonging to $[k]$. We view these symbols as variables. A variable word $w(v)$ over $k$ is a sequence in $[k]\cup\{v\}$, where the variable $v$ occurs at least once. More generally, an $m$-dimensional variable word $w(v_0,...,v_{m-1})$ over $k$ is a sequence in $[k]\cup\{v_0,...,v_{m-1}\}$ such that each $v_j$ occurs at least once and they are in block position, meaning that if $w(v_0,...,v_{m-1})$ is of the form $(x_0,...,x_{n-1})$ then $\max \{i:x_i=v_j\}<\min\{i:x_i=v_{j+1}\}$ for all $0\mik j<m-1$.
Clearly, every variable word can be viewed as an $1$-dimensional variable word.

Let $k,m$ be positive integers and $w(v_0,...,v_{m-1})$ an $m$-dimensional variable word over $k$. For every sequence of symbols $\mathbf{x}=(x_i)_{i=0}^{m-1}$ of length $m$ we denote by $w(\mathbf{x})$ the sequence resulting by substituting each occurrence of $v_i$ by $x_i$ for all $0\mik i<m$. Observe that $w(\mathbf{x})$ is an $m'$-dimensional variable word, for some $m'\mik m$, if and only if $\mathbf{x}$ is an $m'$-dimensional variable word. In particular, $w(\mathbf{x})$ is a variable word if and only if $\mathbf{x}$ is a variable word. An $m'$-dimensional (resp. single) variable word is called reduced by $w(v_0,...,v_{m-1})$ if it is of the form $w(\mathbf{x})$ for some $m$-dimensional (resp. single) variable word $\mathbf{x}$ of length $m$.

\begin{thm}
  \label{gr_base_case} For every triple of positive integers $k,r,m$ there exists a positive integer $n_0$ with the following property. For every integer $n$ with $n\meg n_0$ and every $r$-coloring of all the variable words over $k$ of length $n$, there exists an $m$-dimensional variable word $w(v_0,...,v_{m-1})$ over $k$ of length $n$ such that the set of all variable words over $k$ reduced by $w(v_0,...,v_{m-1})$ is monochromatic. We denote the least such $n_0$ by $GR(k,m,r)$.

  Moreover, the numbers $GR(k,m,r)$ are upper bounded by a primitive recursive function belonging to the class $\mathcal{E}^5$ of Grzegorczyk's hierarchy.
\end{thm}

\section{The Hindman Theorem}
The case ``$k=1$'' of Theorem \ref{gr_base_case} follows by the finite version of Hindman's theorem \cite{H}. To state it we need some pieces of notation.

Let $n,m,d$ be positive integers with $d\mik m\mik n$.
We denote by $\mathcal{F}(n)$ the set of all non-empty subsets of $\{0,...,n-1\}$.
A finite sequence $\mathbf{s}=(s_i)_{i=0}^{m-1}$ in $\mathcal{F}(n)$ is called block if $\max s_i<\min s_{i+1}$
for all $0\mik i<m-1$. We denote the set of all block sequences of length $m$ in $\mathcal{F}(n)$ by $\mathrm{Block}^m(n)$.
For every $\mathbf{s}=(s_i)_{i=0}^{m-1}$ in $\mathrm{Block}^m(n)$ we define the set of nonempty unions of $\mathbf{s}$ to be
\[\mathrm{NU}(\mathbf{s})=\Big\{\bigcup_{i\in t}s_i:t\;\text{is a nonempty subset of}\;\{0,...,m-1\}\Big\}.\]
We say that a block sequence $\mathbf{t}=(t_i)_{i=0}^{d-1}$ in $\mathcal{F}(n)$ is a block subsequence of $\mathbf{s}$ if $t_i\in\mathrm{NU}(\mathbf{s})$ for all $0\mik i<d$.
The finite version of Hindman's theorem is stated as follows.
\begin{thm}
  \label{Hindman_fin}
  For every pair $m,r$ of positive integers, there
exists a positive integer $n_0$ with the following property. For every finite block sequence $\mathbf{s}$ of nonempty
finite subsets of $\nn$
of length at least $n_0$ and every coloring of the set $\mathrm{NU}(\mathbf{s})$ with $r$ colors, there exists a block subsequence $\mathbf{t}$ of $\mathbf{s}$ of length $m$
such that the set $\mathrm{NU}(\mathbf{t})$ is monochromatic.
We denote the least $n_0$ satisfying the above property by $\mathrm{H}(m,r)$.

  Moreover, the numbers $\mathrm{H}(m,r)$ are upper bounded by a primitive recursive function belonging to the class $\mathcal{E}^4$ of Grzegorczyk's hierarchy.
\end{thm}

This finite version follows by the disjoint union theorem \cite{GR, Tay2} and Ramsey's theorem.
The bounds for the disjoint union theorem given in \cite{Tay2}, as well as, the bound for the Ramsey numbers given in \cite{ER} are in $\mathcal{E}^4$. Using these bounds, one can see that the
numbers $\mathrm{H}(m,r)$ are upper bounded by a primitive recursive function belonging to the class $\mathcal{E}^4$ of Grzegorczyk's hierarchy. We refer the interested reader to \cite{DK}.

\section{Insensitivity}
The proof of Theorem \ref{gr_base_case} proceeds by induction on $k$. The main notion that helps us to carry out the inductive step of the proof is an appropriate modification of Shelah's insensitivity (see Definition \ref{defn_ins} below).

First, let us introduce some additional notation.
Let $k,m,n$ be positive integers with $m\mik n$ and $w(v_0,...,v_{m-1})$ be an $m$-dimensional variable word over $k$ of length $n$.
We denote by $W^k_v(n)$ the set of all variable words over $k$ of length $n$, while by
$W^k_v\big(w(v_0,...,v_{m-1})\big)$ the set of all variable words over $k$ reduced by $w(v_0,...,v_{m-1})$. If $w=w(v_0,...,v_{m-1})=(x_i)_{i=0}^{n-1}$, for every $j=0,...,m-1$ we set \[\mathrm{supp}_{w}(v_j)=\{i\in\{0,...,n-1\}:x_i=v_j\}.\]
We consider the following analogue of Shelah's insensitivity.
\begin{defn}
  \label{defn_ins}
  Let $k,m,n$ be positive integers with $m\mik n$. Also let
  $w(v_0,...,v_{m-1})$ be an $m$-dimensional variable word over $k+1$ of length $n$ and $a,b$ in $[k+1]$ with $a\neq b$.
  \begin{enumerate}
    \item[(i)] We say that two  words $\mathbf{x}=(x_i)_{i=0}^{n-1}$ and $\mathbf{y}=(y_i)_{i=0}^{n-1}$ over $k+1$ of length $n$ are $(a,b)$-equivalent if for every $e$ in $[k+1]\setminus\{a,b\}$, we have that $x_i=e$ if and only if $y_i=e$ for all $i$ in
        $\{0,...,n-1\}$.

    \item[(ii)] We say that a coloring $c$ of $W^{k+1}_v(n)$ is $(a,b)$-insensitive over $w(v_0,...,v_{m-1})$ if for every pair $\mathbf{x},\mathbf{y}$ of $(a,b)$-equivalent words over $k+1$ of length $m$, we have that $c\big(w(\mathbf{x})\big)=c\big(w(\mathbf{y})\big)$.
\end{enumerate}
\end{defn}
We prove the following analogue of Shelah's insensitivity lemma.
\begin{lem}
  \label{insensitivity}
   For every triple $k,m,r$ of positive integers there exists a
   positive integer $n_0$ satisfying the following. For every integer $n$ with $n\meg n_0$, every  $a,b$ in $[k+1]$ with $a\neq b$ and every $r$-coloring $c$ of $W_v^{k+1}(n)$ there exists an $m$-dimensional variable word $w(v_0,...,v_{m-1})$ over $k+1$ of length $n$ such that $c$ is $(a,b)$-insensitive over $w(v_0,...,v_{m-1})$.
   We denote the least such $n_0$ by $\mathrm{Sh}_v(k,m,r)$.

  Finally, the numbers $\mathrm{Sh}_v(k,m,r)$ are upper bounded by a primitive recursive function belonging to the class $\mathcal{E}^4$ of Grzegorczyk's hierarchy.
\end{lem}

Before we proceed to the proof of Lemma \ref{insensitivity} let us define a function $f:\nn^4\to\nn$ by the following rule. For every choice of positive integers $k,m,r$ we recursively define
\[
\left\{ \begin{array} {l} f(k,0,m,r)=0,\\
                               f(k,j+1,m,r)=f(k,j,m,r)+r^{ (k+2)^{m-j-1+f(k,j,m,r)}}\end{array}  \right.
\]
and we set $f(k,j,m,r)=0$ if at least one of the integers $k,m,r$ is equal to zero.
Observe that $f$ belongs to the class $\mathcal{E}^4$ of Grzegorczyk's hierarchy.
\begin{proof}
  [Proof of Lemma \ref{insensitivity}]
  Let $k,m,r$ of positive integers.
  We will show the inequality
  \begin{equation}
    \label{eq01}
    \mathrm{Sh}_v(k,m,r)\mik f(k,m,m,r).
  \end{equation}
  Indeed, let $n$ be an integer with $n\meg f(k,m,m,r)$ and $c:W_v(n)\to\{1,...,r\}$.
  Also let $a,b$ in $[k+1]$ with $a\neq b$.
  Set
  \[q_j=f(k,m-j,m,r)+j\]
  for all $0\mik j\mik m$. We inductively construct a sequence $(w_j)_{j=0}^{m}$
  satisfying for every $j=0,...,m$ the following.
  \begin{enumerate}
    \item[(i)] $w_j$ is a $q_j$-dimensional variable word over $k+1$ of length $n$.
    \item[(ii)] If $0<j$, then $w_j$ is reduced by $w_{j-1}$.
    \item[(iii)] If $1<j$, then $\mathrm{supp}_{w_j}(v_{j-2})=\mathrm{supp}_{w_{j-1}}(v_{j-2})$.
    \item[(iv)] If $0<j$, then for every $\mathbf{x}=(x_i)_{i=0}^{q_j-1},\mathbf{y}=(y_i)_{i=0}^{q_j-1}$ in $W^{k+1}_v(q_j)$ such that
        \begin{enumerate}
          \item[(1)] $x_i=y_i$ for all $i=0,...,q_j-1$ with $i\neq j-1$ and
          \item[(2)] $x_{j-1}=a$ and $y_{j-1}=b$,
        \end{enumerate}
        we have that $c\big(w_{j}(\mathbf{x})\big)=c\big(w_{j}(\mathbf{y})\big)$.
  \end{enumerate}
  We pick an arbitrary $q_0$-dimensional variable word $w_0=w_0(v_0,...,v_{q_0-1})$ over $k+1$ of length $n$. Clearly, condition (i) above is satisfied while conditions (ii)-(iv) are meaningless. Let us assume that for some positive $j$ we have constructed $w_0,...,w_{j-1}$ satisfying the conditions above. Set
  $d=r^{(k+2)^{q_j-1}}$ and
  observe that
  \begin{equation}
    \label{eq02}
    q_j+d=q_{j-1}+1.
  \end{equation}
  Moreover, for every $t=0,...,d$ we set
  \[
\mathbf{a}_t=(\underbrace{a,...,a}_{t-\mathrm{times}}, \underbrace{b,...,b}_{(d-t)-\mathrm{times}})
\]
and $A=\{\mathbf{a}_t:t=0,...,d\}$.
We define a map $Q$ from $A\times W^{k+1}_v(q_j-1)$ into $W^{k+1}_v(q_{j-1})$ setting for each $t$ in $\{0,...,d\}$ and $(z_i)_{i=0}^{q_j-2}$ in $W^{k+1}_v(q_j-1)$
\[Q(\mathbf{a}_t,(z_i)_{i=0}^{q_j-2})=(z_i)_{i=0}^{j-2}\;^\con\mathbf{a}_t^\con(z_i)_{i=j-1}^{q_j-2},\]
under the convection that $(z_i)_{i=0}^{j-2}$ (resp. $(z_i)_{i=j-1}^{q_j-2}$) is the empty sequence if $j=1$ (resp. $j=m$).
We denote by $\mathcal{X}$ the set of all maps from $W^{k+1}_v(q_j-1)$ into $\{1,...,r\}$. Clearly, $\mathcal{X}$ is of cardinality at most $d$.
For every $t$ in $\{0,...,d\}$, we define $T_t$ in $\mathcal{X}$ setting for every $\mathbf{z}$ in $W^{k+1}_v(q_j-1)$
\[T_t(\mathbf{z})=c(w_{j-1}(Q(\mathbf{a}_t,\mathbf{z}))).\]
Since the cardinality of $\mathcal{X}$ is at most $d$, there exist $t_1,t_2$ in $\{0,...,d\}$ such that $t_1<t_2$ and $T_{t_1}=T_{t_2}$.
Finally, we set
\[w'(v_0,...,v_{q_j-1})=(v_i)_{i=0}^{j-2}\;^\con(\underbrace{a,...,a}_{t_1-\mathrm{times}},\underbrace{v_{j-1},...,v_{j-1}}_{(t_2-t_1)-\mathrm{times}},\underbrace{b,...,b}_{(d-t_2)-\mathrm{times}})^\con(v_i)_{i=j}^{q_j-1},\]
under the convection that $(v_i)_{i=0}^{j-2}$ (resp. $(v_i)_{i=j}^{q_j-1}$) is the empty sequence if $j=1$ (resp. $j=m$),
and $w_j(v_0,...,v_{q_j-1})=w_{j-1}\big(w'(v_0,...,v_{q_j-1})\big)$.
By equation \eqref{eq02}, we have that $w'$ is of length $q_{j-1}$ and therefore $w_j$ is well defined.
It is immediate that $w_j(v_0,...,v_{q_j-1})$ satisfies conditions (i)-(iii). Let
$\mathbf{x}=(x_i)_{i=0}^{q_j-1},\mathbf{y}=(y_i)_{i=0}^{q_j-1}$ in $W^{k+1}_v(q_j)$ as in condition (iv). Define $\mathbf{z}=(z_i)_{i=0}^{q_j-2}$ setting $z_i=x_i$ if $i<j-1$ and $z_i=x_{i+1}$ otherwise. Observe that
\[
  Q(\mathbf{a}_{t_2},\mathbf{z})=w'(\mathbf{x})\;\text{and}\;Q(\mathbf{a}_{t_1},\mathbf{z})=w'(\mathbf{x}).
\]
Therefore,
\[\begin{split}
  c(w_j(\mathbf{x}))&=
            c(w_{j-1}(w'(\mathbf{x})))
            =c(w_{j-1}(Q(\mathbf{a}_{t_2},\mathbf{z})))
            =T_{t_2}(\mathbf{z})
            =T_{t_1}(\mathbf{z})\\
            &=c(w_{j-1}(Q(\mathbf{a}_{t_1},\mathbf{z})))
            =c(w_{j-1}(w'(\mathbf{y})))
            =c(w_j(\mathbf{x}))
\end{split}\]
as desired and the proof of the inductive step of the construction is complete.

  Let us set $w(v_0,...,v_{m-1})=w_m(v_0,...,v_{m-1})$ and observe that $w(v_0,...,v_{m-1})$ is as desired. Indeed, first observe that by condition (ii) of the inductive construction we have that $w$ is reduced by $w_j$ for all $j$ in $\{0,...,m\}$. Moreover, by condition (ii) we have that $\supp_w(v_{j-1})=\supp_{w_j}(v_{j-1})$ for all $j$ in $\{1,...,m\}$. Thus, for every $j$ in $\{1,...,m\}$ and every $\mathbf{x}'=(x'_i)_{i=0}^{m-1},\mathbf{y}'=(y'_i)_{i=0}^{m-1}$ in $W^{k+1}_v(m)$ such that $x'_{j-1}=a$, $y'_{j-1}=b$ and $x'_i=y'_i$ for all $i\neq j-1$, there exist
  $\mathbf{x}=(x_i)_{i=0}^{q_j-1},\mathbf{y}=(y_i)_{i=0}^{q_j-1}$ in $W^{k+1}_v(q_j)$ satisfying:
  \begin{enumerate}
          \item[(a)] $x_{j-1}=a$ and $y_{j-1}=b$
          \item[(b)] $x_i=y_i$ for all $i=0,...,q_j-1$ with $i\neq j-1$
          \item[(c)] $w(\mathbf{x}')=w_{j}(\mathbf{x})$ and $w(\mathbf{y}')=w_{j}(\mathbf{y})$
        \end{enumerate}
   and therefore, by condition (iv) we have that
   \begin{equation}
     \label{eq03}
     c(w(\mathbf{x}'))\stackrel{\text{(c)}}{=}c(w_{j}(\mathbf{x}))\stackrel{\text{(iv)}}{=}c(w_{j}(\mathbf{y}))\stackrel{\text{(c)}}{=}c(w(\mathbf{y}')).
   \end{equation}
  One can easily see that \eqref{eq03} implies that the coloring $c$ in $(a,b)$-insensitive over $w(v_0,...,v_{m-1})$.
  Thus inequality \eqref{eq01} is valid and since $f$ belongs to the class $\mathcal{E}^4$ of Grzegorczyk's hierarchy, the proof of the lemma is complete.
\end{proof}

\section{Proof of Theorem \ref{gr_base_case}}
As we mentioned in the introduction, the proof of Theorem \ref{gr_base_case} is a modification of S. Shelah's proof for the Hales--Jewett Theorem. It proceeds by induction on $k$. For ``$k=1$'' Theorem \ref{gr_base_case} follows readily by the finite version of Hindman's theorem, that is, Theorem \ref{Hindman_fin}. In particular, we have
\begin{equation}
  \label{eq04}
  \mathrm{GR}(1,m,r)=\mathrm{H}(m,r).
\end{equation}
Towards the proof of the inductive step, we, in particular, show the following inequality.
\begin{equation}
  \label{eq05}
  \mathrm{GR}(k+1,m,r)\mik \mathrm{Sh}_v(k,\mathrm{GR}(k,m,r),r).
\end{equation}
Indeed, let us set $M=\mathrm{GR}(k,m,r)$ and pick any integer $n$ with $n\meg\mathrm{Sh}_v(k,M,r)$. Also, let $c$ be an $r$-coloring of $W^{k+1}_v(n)$.
By Lemma \ref{insensitivity}, there exists an $M$-dimensional variable word $w'(v_0,...,v_{M-1})$ over $k+1$ of length $n$ such that the coloring $c$ is $(k,k+1)$-insensitive.
We define an $r$-coloring $c'$ on $W_v^k(M)$ by setting
\[c'(\mathbf{x})=c(w'(\mathbf{x}))\]
for all $\mathbf{x}$ in $W_v^k(M)$. By the definition of $M$, there exists an $m$-dimensional variable word $w''(v_0,...,v_{m-1})$ over $k$ of length $M$ such that the set $W_v^k(w'')$ is $c'$-monochromatic. We set $w(v_0,...,v_{m-1})=w'(w''(v_0,...,v_{m-1}))$. Clearly, $w$ is reduced by $w'$ and therefore, since $c$ is $(k,k+1)$-insensitive over $w'$, we have that $c$ is $(k,k+1)$-insensitive over $w$ too. Moreover, by the definition of $c'$ and the choice of $w''$, we have that the set $\{w(\mathbf{x}):\mathbf{x}\in W_v^k(m)\}$ is $c$-monochromatic. Invoking the insensitivity of the coloring $c$ over $w$ we have that $W^{k+1}_v(w)$ is monochromatic as desired.

Finally, by \eqref{eq04},\eqref{eq05} and the fact that both the numbers $\mathrm{H}(m,r)$ and  $\mathrm{Sh}_v(k,m,r)$ are upper bounded by a primitive recursive function belonging to the class $\mathcal{E}^4$ of Grzegorczyk's hierarchy, we have that the numbers $\mathrm{GR}(k,m,r)$ are upper bounded by a primitive recursive function belonging to the class $\mathcal{E}^5$.


\end{document}